\documentclass[a4paper]{amsart}
\usepackage{amsmath,amsthm}
\usepackage{amsfonts}
\usepackage{latexsym}
\usepackage{amsxtra}
\usepackage{amscd}
\usepackage{graphicx}
\usepackage{pstricks,pst-node,pst-text,pst-coil, pst-plot, pst-grad}

\numberwithin{equation}{section}

\newtheorem{theorem}{Theorem}[section]
\newtheorem{corollary}[theorem]{Corollary}
\newtheorem{lemma}[theorem]{Lemma}
\newtheorem{example}[theorem]{Example}
\newtheorem{proposition}[theorem]{Proposition}
\newtheorem{remark}[theorem]{Remark}

\DeclareMathOperator{\Ker}{\mathrm{ker}}

\newcommand{\R}{\mathbb{R}}

\newcommand{\Z}{\mathbb{Z}}
\newcommand{\F}{\mathcal{F}}
\newcommand{\PP}{\boldsymbol{P}}
\newcommand{\mm}{\boldsymbol{m}}
\newcommand{\kk}{\boldsymbol{k}}
\newcommand{\Prj}{\mathcal{P}}
\newcommand{\Oc}{\mathcal{O}}
\newcommand{\sign}{\mathop\mathrm{sign}\nolimits}

\newcommand{\s}{\mathfrak{s}}
\newcommand{\gb}{\boldsymbol{g}}
\renewcommand{\t}[1]{\widetilde{#1}}
\newcommand{\D}{\partial}
\newcommand{\dif}{\mathrm{\,d}}
\DeclareMathOperator{\Pr1}{\mathrm{Pr}_1}
\newcommand{\f}{\mathfrak{f}}
\newcommand{\g}{\mathfrak{g}}
\newcommand{\h}{\mathfrak{h}}
\newcommand{\E}{\mathcal{E}}

\begin{document}

\author{Alessandro Calamai}
\address[Alessandro Calamai]{Dipartimento di Ingegneria Industriale e Scienze Matematiche, 
Universit\`a Politecnica delle Marche, Via Brecce Bianche, 60131 Ancona, Italy}
\author{Marco Spadini}
\address[Marco Spadini]{Dipartimento di Matematica e Informatica `U.\ Dini', Universit\`a 
di Firenze, Via Santa Marta 3, 50139 Firenze, Italy}
\title[Periodic perturbations on implicitly defined manifolds]{Periodic 
perturbations of constrained motion problems on a class of implicitly defined manifolds}

\begin{abstract}
We study forced oscillations on differentiable manifolds which are globally defined as 
the zero set of appropriate smooth maps in some Euclidean spaces. Given a $T$-periodic 
perturbative forcing field, we consider the two different scenarios of a nontrivial 
unperturbed force field and of perturbation of the zero field. We provide simple, 
degree-theoretic conditions for the existence of branches of $T$-periodic solutions. We 
apply our construction to a class of second order Differential-Algebraic Equations.
\end{abstract}


\maketitle

\section{Introduction}

In this paper we study $T$-periodic solutions of some parametrized families of 
$T$-periodic constrained second order Ordinary Differential Equations (ODEs). In other 
words, we study forced oscillations on a differentiable submanifold of some Euclidean space.
As a physical interpretation, the equations we consider here represent the motion equations 
of a constrained system, the manifold being the constraint. We work under the 
assumption that such a manifold is globally defined as the zero set of a $C^\infty$-smooth 
map and, given a $T$-periodic perturbative forcing field, we consider two different 
scenarios according whether there exists a nontrivial unperturbed force field. Namely, let
$g\colon U\subseteq\R^m\times\R^s\to\R^s$ be a $C^\infty$ map  such that the Jacobian matrix 
of $g$ with respect to the last $s$ variables, $\partial_2g(x,y)$, is nonsingular for all 
$(x,y)\in U$.
This implies that $0\in\R^s$ is a regular value of $g$ so that $M=g^{-1}(0)$ is a 
$C^\infty$-smooth submanifold of $\R^k=\R^m\times\R^s$ of dimension $m$ (and codimension $s$).

For second order ODEs on differentiable manifolds we adopt the notation of e.g.\ 
\cite{Fu94}. We consider parametrized second order ODEs on $M$ that, with this notation,
assume the following forms:
\begin{subequations}\label{unoedue}
\begin{equation}
\ddot \xi_\pi=f(\xi,\dot\xi)+\lambda h(t,\xi,\dot\xi)  \label{uno}
\end{equation}
and
\begin{equation}\label{due}
\ddot \xi_\pi=\lambda h(t,\xi,\dot\xi),
\end{equation}
\end{subequations}
where $\lambda \geq 0$ is a parameter, $\xi(t)\in M$,
$\ddot \xi_\pi$ stands for the tangential part of the acceleration, $h\colon \R\times TM\to\R^k$ 
and $f\colon TM\to\R^k$ are continuous maps with the property that $f(\xi,\eta)$ and $h(t,\xi,\eta)$ 
belong to $T_\xi M$ for any $(t,\xi,\eta)\in\R\times TM$, and $h$ is $T$-periodic in the first 
variable.

Notice that the assumptions on $M$ and $g$ imply that, locally, the manifold $M$ can be 
represented as graph of some map {}from an open subset of $\R^{m}$ to $\R^s$. Thus, writing
$\xi(t)=\big(x(t),y(t)\big)$, Equations \eqref{unoedue} can be locally simplified via the 
implicit function theorem. In view of this fact one might think that it is possible to reduce 
Equations \eqref{unoedue} to ordinary differential equations in $\R^{m}$. It is not so. In fact, 
globally, $M$ may not be the graph of a map {}from an open subset of $\R^{m}$ to $\R^s$ as, for 
instance, when $U=\R^3$ and $g\colon \R\times\R^2\to\R^2$ is given by
\[
g(x,y)=g(x;y_1,y_2)=\big(e^{y_1}\cos(y_2)-x, e^{y_1}\sin(y_2)-x\big).
\]
In this case, although $\det\partial_2g(x,y)\neq 0$, one clearly has that the 
$1$-dimensional manifold $M=g^{-1}(0)$ is not the graph of a function 
$x\mapsto\big(y_1(x),y_2(x)\big)$. In fact, $M$ consists of infinitely many 
connected components each lying in a plane $y_2=\frac{\pi}{4}+l\pi$ for $l\in\Z$.

Observe also that even when $M$ is a (global) graph of some map $\Gamma$, the expression 
of $\Gamma$ might be too complicated to use or impossible to determine analytically so 
that, the simplified versions of Equations \eqref{unoedue} may be too difficult to use. A 
simple example of this fact is obtained by taking $m=s=1$, $U=\R\times\R$ and 
$g(x,y)=y^7+y-x^2+x^5$.

\smallskip
A pair $(\lambda,\xi)$, with $\lambda \geq 0$ and $\xi\colon  \R \to M$ a $T$-periodic solution
of \eqref{uno} (resp., of \eqref{due}) corresponding to $\lambda$, is called a
\emph{solution pair} of  \eqref{uno} (resp., of \eqref{due}).
The set of solution pairs is regarded as a subset of $[0,+\infty)\times C^1_T(M)$, where 
$C^1_T(M)$ is the metric subspace of the Banach space $C^1_T(\R^k)$ of the
$T$-periodic $C^1$ maps {}from $\R$ to $M$. In this paper we investigate the structure of 
the set of solution pairs of Equations \eqref{unoedue}. Namely, we will prove global 
continuation results for solution pairs of \eqref{unoedue}.

Our results here stem {}from a combination of those of \cite{FP93,FS99} about branches of 
$T$-periodic solutions of constrained second order equations, with the formula contained in 
\cite{CaSp} for the computation of the degree of a tangent vector field in terms of
the Brouwer degree of an appropriate map.
In the situation considered in this paper (namely, for manifolds which are of the form 
$M=g^{-1}(0)$ with $g$ as above), our results are complementary and, in some sense,
improve upon those of \cite{FP93,FS99} since the Brouwer degree of the considered maps, having 
a simpler nature than the degree of a tangent vector field, is in principle easier to compute. 
In fact, the degree of a tangent vector field can be seen as an extension of the notion of 
Brouwer degree (a different one being the notion of degree of maps between oriented manifolds). 
Actually, since vector fields in Euclidean spaces can be regarded as maps and vice versa, the 
degree of a vector field is essentially the Brouwer degree, with respect to $0$, of the field 
seen as a map.

\smallskip
A remarkable application of our results, which justifies our interest in the particular
differentiable manifolds considered in this paper, is the content of the last section 
where we study periodic perturbations of a particular class of second order Differential-Algebraic 
Equations (DAEs) in semi-explicit form. 

Recently, DAEs have received increasing interest due, in particular, to applications in 
engineering and have been the subject of extensive study (see e.g.\ \cite{KM} for a 
comprehensive treatment) aimed mostly (but not only) to numerical methods. Our approach here, 
as in our recent works \cite{Ca10, CaSp, Spa}, is directed towards qualitative theory of some 
particular DAEs which are studied by means of topological methods, making use of the 
equivalence of the given equations and suitable ODEs on manifolds. In particular, in this paper 
we investigate particular DAEs of second order, whereas the papers \cite{Ca10, CaSp, Spa}
were devoted to the first order case.
The type of equations that we study here may be used to represent some 
nontrivial physical systems, as, for instance, constrained systems
(see e.g.\ \cite{RR}).
As we will prove, such equations are equivalent to second
order ODEs on precisely the kind a manifold that we consider in the first part of the 
paper. This fact enables us to take advantage of the continuation results in \cite{FP93,FS99} 
and to find some surprisingly simple formulas useful for the study of connected sets of 
forced oscillations when a periodic forcing term is introduced. Some applications of our 
results are shown in a few illustrative examples.

\smallskip
We observe that our results, besides their intrinsic interest, can be useful in order to 
establish existence theorems in presence of a priori bounds (for a general reference on 
continuation methods, see e.g.\ \cite{N}). In addition, they can be used to get some 
`topological' multiplicity theorems for forced oscillations as in \cite{FPS00}. We will not 
pursue these lines here, though.

Finally, we wish to point out that in this paper we make use of deep results which are based 
on the fixed point index, but our techniques require just the notion of the well-known Brouwer 
degree. Our effort has been to make the paper accessible also to researchers which are not
particularly familiar with topological methods.

\section{Preliminaries}

\subsection{Tangent vector fields and the notion of degree}

We now recall some basic notions about tangent vector fields on manifolds as well as the 
notion of degree of an admissible tangent vector field.

Let $M\subseteq\R^k$ be a manifold. Let $w$ be a tangent vector field on $M$, that 
is, a continuous map $w\colon M\to\R^k$ with the property that $w(\xi)\in T_{\xi}M$ for any 
$\xi\in M$. 
If $w$ is (Fr\'echet) differentiable at $\xi\in M$ and $w(\xi)=0$, then the differential 
$dw_{\xi} \colon  T_{\xi}M\to\R^k$ maps $T_{\xi}M$ into itself (see e.g.\ \cite{Mi}), so that 
the determinant $\det dw_{\xi}$ of $dw_{\xi}$ is defined. If, in addition, $\xi$ is a 
nondegenerate zero (i.e.\ $dw_{\xi} \colon  T_{\xi}M\to\R^k$ is injective) then $\xi$ is an 
isolated zero and $\det dw_{\xi}\neq 0$. 

Let $W$ be an open subset of $M$ in which we assume $w$ \emph{admissible} (for the degree); that 
is, the set $w^{-1}(0) \cap W$ is compact. Then, one can associate to the pair $(w,W)$ an integer, 
$\deg(w,W)$, called the \emph{degree (or characteristic) of the vector field $w$ in $W$}, which, 
in a sense, counts (algebraically) the zeros of $w$ in $W$ (see e.g.\ \cite{FPS05, H, Mi} and 
references therein). In fact, when the zeros of $w$ are all nondegenerate, then the set 
$w^{-1}(0)\cap W$ is finite and
\begin{equation}\label{sommasegni}
\deg(w,W)=\sum_{\xi\in w^{-1}(0)\cap W}{\rm sign} \det dw_{\xi}.
\end{equation}
Observe that in the flat case, i.e.\ when $M = \R^k$, $\deg(w,W)$ is just the classical
Brouwer degree with respect to zero, $\deg_B(w,V,0)$, where $V$ is any
bounded open neighborhood of $w^{-1}(0) \cap W$ whose closure is
contained in $W$.

The notion of degree of an admissible tangent vector field plays a crucial role throughout
this paper. 
All the standard properties of the Brouwer degree for continuous
maps on open subsets of Euclidean spaces, such as homotopy invariance,
excision, additivity, existence, still hold in this more general
context (see e.g.\ \cite{FPS05}).

\begin{remark}\label{excision}
The Excision Property allows the introduction of the notion of index of an isolated zero 
of a tangent vector field. Indeed, let $\xi\in M$ be an isolated zero of $w$. Clearly, $\deg(w,V)$ 
is well defined for each open $V\subseteq M$ such that $V\cap w^{-1}(0)=\{ \xi\}$. By the 
Excision Property $\deg(w,V)$ is constant with respect to such $V$'s. This common value of 
$\deg(w,V)$ is, by definition, the \emph{index of $w$ at $ \xi$}, and is denoted by 
$\mathrm{i}\,(w, \xi)$. Using this notation, if $(w,W)$ is admissible, by the Additivity 
Property we get that, if all the zeros in $W$ of $w$ are isolated, then
\begin{equation}\label{sommaindici}
\deg(w,W)=\sum_{\xi\in w^{-1}(0)\cap W} \mathrm{i}\,(w,\xi).
\end{equation}
By formula \eqref{sommasegni} we have that, if $\xi$ is a nondegenerate zero of $w$, then
\begin{equation*}
 \mathrm{i}\,(w,\xi)=\sign\det d w_{\xi}.
\end{equation*}
Notice that \eqref{sommasegni} and \eqref{sommaindici} differ in the fact that, in the
latter, the zeros of $w$ are not necessarily nondegenerate as they have to be in the former. 
In fact, in \eqref{sommaindici}, $w$ need not be differentiable at its zeros.
\end{remark}

\subsection{Tangent vector fields on implicitly defined manifolds}\label{sectan}

Let $\Psi\colon \R\times M\to\R^k$ be a (time-dependent) tangent vector field on $M\subseteq\R^k$ 
that is, a continuous map with the property that $\Psi(t,\xi)\in T_{\xi}M$ for each 
$(t,\xi)\in\R\times M$. Assume that there is a connected open subset $U$ of $\R^k$ and a 
$C^\infty$ map $g\colon U\to\R^s$ with the property that $M=g^{-1}(0)$. Suppose that up to 
change of coordinates, writing $\R^k=\R^m\times\R^s$, $m=k-s$, one has that the partial 
derivative of $g$ with respect to the second variable, $\partial_2 g(x,y)$, is invertible 
for each $(x,y)\in U$. 

According to the above decomposition of $\R^k$, we can write, for any $\xi\in\R^k$,
$\xi=(x,y)$ and, with a slight abuse of notation, for any $t\in\R$ 
\[
 \Psi(t,\xi)=\Psi(t,x,y)=\big(\Psi_1(t,x,y),
                              \Psi_2(t,x,y)\big).
\]
Notice that one must have 
\begin{equation}\label{phi2}
\Psi_2(t,x,y)=-(\partial_2g(x,y))^{-1}\partial_1g(x,y)\Psi_1(t,x,y).
\end{equation}
In fact, $\Psi(t,\xi)\in T_{\xi}M$ being equivalent to $\Psi(t,\xi)\in\Ker g'(x,y)$,
one has for each $(t,x,y)\in\R\times M$ that
\[
0= g'(x,y)\Psi(t,x,y)=
   \partial_1g(x,y)\Psi_1(t,x,y)+\partial_2g(x,y)\Psi_2(t,x,y),
\]
which implies \eqref{phi2}; here $g'(x,y)$ denotes the Fr\'echet differential of $g$ at 
$(x,y)$.

We now give a formula for the degree of tangents vector fields on $M$ in terms of (potentially 
easier to compute) degree of appropriate vector fields on $U$. 

\begin{remark}\label{remext}
Assume that $\psi\colon  M \to \R^k$ is a tangent vector field on $M$. Since $M=g^{-1}(0)$ is a 
closed subset of the metric space $U$, the well-known Tietze's Theorem (see e.g.\ \cite{Dug}) 
implies that there exists an extension $\t\psi\colon U\to\R^k$ of $\psi$. 
\end{remark}

Remark{ }\ref{remext} shows that it is not restrictive to assume, as we sometimes do, that the 
given tangent vector fields are actually defined on a convenient neighborhood of the manifold 
$M$. In fact, although an arbitrary extension of $\psi$ may have many zeros outside $M$, we 
are interested in the degree of $\psi$ on $M$ which only takes into account those zeros of 
$\psi$ that lie on $M$.

The following result (see \cite{CaSp}, \cite{Spa}) gives a formula for the degree of
tangents vector fields on $M$ in terms of degree of appropriate vector fields on $U$.

\begin{theorem}\label{formuladeg}
Let $U\subseteq\R^m\times\R^s$ be open and connected, let $g\colon U\to\R^s$ be a $C^\infty$ 
function such that $\partial_2g(x,y)$ is nonsingular for any $(x,y)\in U$ and let $M=g^{-1}(0)$. 
Assume that $\psi\colon M\to\R^m\times\R^s$ is tangent to $M$ and let $\t\psi_1$ be the 
projection on $\R^k$ of an arbitrary continuous extension $\t\psi$ of $\psi$ to $U$. Define 
$\F\colon U\to\R^m\times\R^s$ by $\F(x,y)=\big(\t\psi_1(x,y),g(x,y)\big)$. Then, $\F$ is 
admissible in $U$ if and only if so is $\psi$ in $M$, and
\begin{equation}\label{idgradi}
 \deg(\psi,M)=\s\deg(\F,U),
\end{equation}
where $\s$ is the constant sign of $\det\partial_2g(x,y)$ for all $(x,y)\in U$.
\end{theorem}

\section{Second order equations on manifolds}\label{secsec}

As in the previous section, let $M$ be a submanifold of $\R^k$. Given a continuous map
$\varphi \colon \R\times TM\to\R^k$ such that $\varphi (t,\xi,\eta) \in T_\xi M$ for any 
$(t,\xi,\eta)\in\R\times TM$, we will say that \emph{$\varphi$ is tangent to $M$}.

It is known (compare \cite{Fu94})
that the \emph{motion equation} associated with the \emph{force} $\varphi$ can be written 
in the form
\begin{equation}\label{secord}
\ddot \xi_\pi=\varphi (t,\xi,\dot\xi),
\end{equation}
where $\ddot \xi_\pi$ stands for the \emph{parallel component of the acceleration} 
$\ddot \xi \in \R^k$ at the point $\xi$. Namely, $\ddot \xi_\pi$ denotes the orthogonal 
projection of $\ddot \xi$ onto $T_{\xi}M$. A \emph{solution} of \eqref{secord} is a $C^2$ 
map $\xi\colon J \to M$, defined on a nontrivial interval $J$, such that 
$\ddot \xi_\pi (t) =\varphi (t,\xi(t),\dot\xi(t))$ for all $t \in J$.

Further, Equation \eqref{secord} can be written in an equivalent way as a first order 
equation on the tangent bundle $TM$ in the form 
\begin{equation}
\dot \zeta =\Psi(t,\zeta),  \label{fib2}
\end{equation}
where, for $\zeta =(\xi,\eta)$, 
\begin{equation}\label{fib3}
\Psi(t;\xi,\eta) =\big(\eta,r(\xi,\eta)+\varphi(t;\xi,\eta)\big), 
\end{equation}
and the map $r\colon TM\to\R^k$ is smooth, quadratic in the second variable 
$\eta\in T_{\xi}M$ for any $\xi\in M$, and with values in $(T_{\xi}M)^{\perp}$.
Such a map $r$ is strictly related to the second fundamental form on 
$M$ and may be interpreted as the \emph{reactive force} due to the constraint $M$.
Actually $r(\xi,\eta)$ is the unique vector in $\R^k$ which makes $\big(\eta,r(\xi,\eta)\big)$ 
tangent to $TM$ at $(\xi,\eta)$. It is well known that $\Psi$, called 
the \textit{second order vector field associated to} $\varphi$, is a (time-dependent)
tangent vector field on $TM$. Hence \eqref{fib2} is actually a first order equation 
on $TM$.

\smallskip
We now focus on the computation of the map $r$.

Let $\xi\colon J\subseteq\R\to M$ be a local solution of \eqref{secord} taking values in 
a neighborhood $V\subseteq M$ of $\xi_0:=\xi(t_0)$ for some $t_0\in J$. Restricting 
$V$ if necessary, we can assume that there exists $U\subseteq\R^k$ and a $C^\infty$ 
function $G\colon U\to\R^s$ such that $U\cap M=V=G^{-1}(0)$. Differentiating twice at 
$t_0$ the relation $G(\xi(t))=0$, we get
\begin{equation}\label{sodiff}
 G''(\xi_0)\big(\dot \xi_0,\dot\xi_0\big)+G'(\xi_0)\ddot \xi(t_0)=0,
\end{equation}
where, for the sake of simplicity, we write $\dot \xi_0$ instead of $\dot \xi(t_0)$.
According to Equations \eqref{fib2}-\eqref{fib3} 
\begin{equation}\label{phipiur}
 \ddot \xi(t_0)=\varphi\big(t_0,\xi_0,\dot \xi_0\big)
                        +r\big(\xi_0,\dot\xi_0\big). 
\end{equation}
Since $\ddot\xi_\pi(t_0)=\varphi\big(t_0,\xi_0,\dot \xi_0\big)\in T_{\xi_0}M$, 
one has $G'(\xi_0)\varphi\big(t_0,\xi_0,\dot\xi_0\big)=0$. Hence \eqref{sodiff} 
yields
\[
G'(\xi_0)r\big(\xi_0,\dot\xi_0\big)=
                      -G''(\xi_0)\big(\dot\xi_0,\dot\xi_0\big).
\]
And, finally,
\begin{equation}\label{formr}
r\big(\xi_0,\dot\xi_0\big)=
   -\left(G'(\xi_0)|_{(T_{\xi_0}M)^\perp}\right)^{-1}
                  G''(\xi_0)\big(\dot\xi_0,\dot\xi_0\big).
\end{equation}
Recall, in fact, that $G'(\xi_0)|_{(T_{\xi_0}M)^\perp}$ is invertible because 
$T_{\xi_0}M$ is the kernel of $G'(\xi_0)$.

Observe that equations \eqref{secord}--\eqref{phipiur} have another 
immediate interesting consequence. Namely, that if $t\mapsto\zeta(t)$ is an $M$-valued 
$C^2$ curve then 
\begin{equation}\label{prjxi}
 \Prj_{\zeta(t)}^\perp \ddot\zeta(t)=r\big(\zeta(t),\dot\zeta(t)\big),
\end{equation}
where $\Prj_{\zeta(t)}^\perp$ denotes the orthogonal projection of $\R^k$ onto the orthogonal
complement to $T_{\zeta(t)}M$.

\smallskip
When $M$ is of the form considered in Section \ref{sectan} we are able to get a more 
explicit expression for the inertial reaction $r$. The expression that we are going to 
obtain, although strictly speaking not necessary for the carrying 
out of our arguments, can be useful for understanding our geometric setting. The reader 
may skip the remaining part of this section at a first reading.

Assume, as in Section \ref{sectan}, that $M=g^{-1}(0)$ where $g\colon U\to\R^s$ is a 
$C^\infty$ map with $\partial_2 g(x,y)$ is invertible for each $(x,y)\in U$. Let 
$\xi\colon J\subseteq\R\to M$ be a local solution of \eqref{secord} and write 
$\xi=(x,y)$. Setting $\xi_0=(x_0,y_0)$ and $\dot\xi_0=(\dot x_0,\dot y_0)$, Equation 
\eqref{sodiff} becomes
\begin{equation}\label{sodiff2}
 g''(x_0,y_0)\big((\dot x_0,\dot y_0),(\dot x_0,\dot y_0)\big)+
\D_1g(x_0,y_0)\ddot x(t_0)+\D_2g(x_0,y_0)\ddot y(t_0)=0.
\end{equation}
Thus, $r(x_0,y_0;\dot x_0,\dot y_0)$ is the (unique) solution lying in 
$\big(T_{(x_0,y_0)}M\big)^\perp$ of the linear equation
\begin{equation}\label{sodiff2bis}
g'(x_0,y_0)r(x_0,y_0;\dot x_0,\dot y_0)=
                      -g''(x_0,y_0)\big((\dot x_0,\dot y_0),(\dot x_0,\dot y_0)\big).
\end{equation}
Now, to find the explicit expression of $r$ we need two elementary linear algebra lemmas.
The proof of the first is straightforward and is left to the reader.

\begin{lemma}
Let $A$ be an $s\times m$ matrix and $B$ be an $s\times s$ matrix.
Let $W=\{(u,v) \in \R^m\times \R^s : Au+Bv=0\}$.
Assume $B$ invertible. Then,
$W^\perp=\{(u,v) \in \R^m\times \R^s : u=A^T(B^{-1})^Tv\}$.
\end{lemma}

\begin{lemma} \label{lemma-invert}
Let $A$ be an $s\times m$ matrix and $B$ be an $s\times s$ matrix.
Assume $B$ invertible. Then, the $s\times s$ matrix
$C=AA^T(B^{-1})^T +B$ is invertible too.
\end{lemma}

\begin{proof}
To prove that $C$ is invertible we will show that so is $B^{-1}C$.
To see this, we prove that $B^{-1}C$ is symmetric and positive definite.
Notice that 
\[
B^{-1}C = B^{-1}AA^T(B^{-1})^T +I=B^{-1}A \big(B^{-1}A\big)^T+I,
\]
thus $B^{-1}C$ is of the form $XX^T+I$ with $X$ a $s\times m$ matrix, $I$ being the identity.
In particular, $B^{-1}C$ is symmetric.
Now, observe that $XX^T$ is always symmetric and positive semidefinite and thus 
$B^{-1}C$ is positive definite, being the sum of a positive definite matrix (the 
identity) and a positive semidefinite matrix.
\end{proof}

Set now $A=\D_1g(x_0,y_0)$, $B=\D_2g(x_0,y_0)$, and $C=AA^T(B^{-1})^T +B$
(recall that $B$ is an invertible $s\times s$ matrix by assumption).
Further, let
$r=(u,v)$ and $\sigma=-g''(x_0,y_0)\big((\dot x_0,\dot y_0),(\dot x_0,\dot y_0)\big)$, so 
that Equation \eqref{sodiff2bis} can be written as
\begin{equation}\label{linear}
Au+Bv=\sigma .
\end{equation}
Taking into account the above lemmas, simple calculations show that 
Equation \eqref{linear}
has a unique solution $(u,v)\in \R^m\times \R^s$ that lies in $\big(T_{(x_0,y_0)}M\big)^\perp$
and is given by
\[
  \left\{
  \begin{array}{l}
   u = A^T(B^{-1})^TC^{-1}\sigma,\\
   v = C^{-1}\sigma.
  \end{array}\right.
\]
Thus, $r(x_0,y_0;\dot x_0,\dot y_0)$ has the following expression:
\begin{equation} \label{explicit}
 r(x_0,y_0;\dot x_0,\dot y_0)=\begin{pmatrix}
           A^T(B^{-1})^TC^{-1}\sigma\\
           C^{-1}\sigma
          \end{pmatrix},
\end{equation}
where $\sigma=-g''(x_0,y_0)\big((\dot x_0,\dot y_0),(\dot x_0,\dot y_0)\big)$.

\begin{example}\label{parabola1}
 Let $k=2$, $m=1$ and $g(x,y)=\frac{1}{2}x^2-y-2$. We wish to compute the map $r$ in this 
 case. Clearly $M=g^{-1}(0)$ is a parabola. With the above notation, we have
 \[
  A=\partial_1g(x,y)=x,\quad B=\partial_2g(x,y)=-1,
                                    \quad C=AA^T(B^{-1})^T+B=1-x^2.
 \]
Also,
\[
 \sigma=-g''(x,y)\left(\left(\begin{smallmatrix}
  \dot x \\ \dot y
 \end{smallmatrix}\right),
 \left(\begin{smallmatrix}
  \dot x \\ \dot y
 \end{smallmatrix}\right)\right)=
 -\left(\begin{smallmatrix}
  \dot x & \dot y
 \end{smallmatrix}\right)\begin{pmatrix}
                          1 & 0\\
                          0 & 0
                         \end{pmatrix}
 \left(\begin{smallmatrix}
  \dot x \\ \dot y
 \end{smallmatrix}\right)
=-\dot x^2.
\]
Therefore,
\[
 r(x,y;\dot x, \dot y)=\begin{pmatrix}
           A^T(B^{-1})^TC^{-1}\sigma\\
C^{-1}\sigma
          \end{pmatrix}=
          \begin{pmatrix}
           \frac{x\dot x^2}{1-x^2} \\ \frac{\dot x^2}{x^2-1}
          \end{pmatrix} .
\]
\end{example}

\begin{example}
 Let $k=3$, $m=2$ and $g(x,y)=z-x^2-y^2$. We wish to compute the map $r$ in this 
 case. Clearly $M=g^{-1}(0)$ is a paraboloid. With the above notation, we have
 \begin{gather*}
  A=\partial_1g(x,y)=\left(\begin{smallmatrix}
  -2x & -2y\end{smallmatrix}\right), \quad B=\partial_2g(x,y)=1,\\ 
  C=AA^T(B^{-1})^T+B=4x^2+4y^2+1.
 \end{gather*}
Also,
\[
 \sigma=-g''(x,y,z)\left(\left(\begin{smallmatrix}
  \dot x \\ \dot y \\ \dot z
 \end{smallmatrix}\right),
 \left(\begin{smallmatrix}
  \dot x \\ \dot y \\ \dot z
 \end{smallmatrix}\right)\right)=-\left(
 \begin{smallmatrix}
  \dot x & \dot y & \dot z
 \end{smallmatrix}\right)\begin{pmatrix}
                          -2 & 0 & 0\\
                          0 & -2 & 0\\
                          0 & 0 & 0
                         \end{pmatrix}
 \left(\begin{smallmatrix}
  \dot x \\ \dot y \\ \dot z
 \end{smallmatrix}\right)
=2\dot x^2+2\dot y^2.
\]
Therefore,
\[
 r(x,y,z;\dot x, \dot y, \dot z)=\begin{pmatrix}
           A^T(B^{-1})^TC^{-1}\sigma\\
C^{-1}\sigma
          \end{pmatrix}=
          \begin{pmatrix}
           \dfrac{-4x(\dot x^2+\dot y^2)}{1+4x^2+4y^2} \\[3mm] 
           \dfrac{-4y(\dot x^2+\dot y^2)}{1+4x^2+4y^2} \\[3mm]
           2\dot x^2+2\dot y^2
          \end{pmatrix}
\]
\end{example}

The examples above, in conjunction with Equations \eqref{fib2} and \eqref{fib3}, show that 
even for simple manifolds second order equations can be rather complicated. The following
example confirms this fact, strengthening the case for the indirect methods investigated 
in this paper.

\begin{example}\label{mostro}
  Let $k=3$, $m=1$ and $g(x,y,z)=\big(z^3+z-x,z-y+x^2\big)$. We wish to compute the map 
$r$ in this case. Here $M=g^{-1}(0)$ is a smooth curve. With the above notation, we have
 \[
  A=\partial_1g(x,y)=\begin{pmatrix}
  -1 \\ 2x\end{pmatrix},
  \quad B=\partial_2g(x,y)=\begin{pmatrix}
                    0  & 3z^2+1\\
                    -1 & 1
                           \end{pmatrix}.
 \]
Also,
\[
 \sigma=-g''(x,y,z)\left(\left(\begin{smallmatrix}
  \dot x \\ \dot y \\ \dot z
 \end{smallmatrix}\right),
 \left(\begin{smallmatrix}
  \dot x \\ \dot y \\ \dot z
 \end{smallmatrix}\right)\right)=
 - \begin{pmatrix}
  \dot 6\dot z^2 \\ 2\dot x^2
 \end{pmatrix}.
\]
Lengthy but straightforward computations yield
\begin{multline*} 
r(x,y,y;\dot x, \dot y, \dot z)=
\frac{2(1+4x^2)}{5+20x^2+16x^4+18z^4+36z^4x^2+12z^2+24z^2x^2}\cdot\\[2mm]
\cdot\begin{pmatrix}
(18xz^4+3z^2+12xz^2+8x^3+4x+1)\dot x^2-(18xz^2+6x+12x^2+9)\dot z^2\\  
+9\dot z^2z^2+3\dot z^2-2\dot x^2-4\dot x^2x^2-9\dot x^2z^4-6\dot x^2z^2\\
18\dot z^2z^2+6\dot z^2+36\dot z^2z^2x^2+12\dot z^2x^2+\dot x^2+4\dot x^2x^2
\end{pmatrix} .
\end{multline*}
\end{example}

\subsection{Branches of harmonic solutions}

In what follows, if $N$ is a differentiable manifold embedded in some $\R^k$, we will 
denote by $C_T^1(N)$, the metric subspace of the Banach space 
$\big( C_T^1(\R^k)\,,\,|\cdot|_1\big)$ of all the $T$-periodic $C^1$ maps $\xi\colon\R\to N$ 
with the usual $C^1$ norm.

In this section we recall two known results {}from \cite{FP93} and \cite{FS99} about the 
sets of \emph{solution pairs} of \eqref{uno} and of \eqref{due}, i.e.\ of those pairs 
$(\lambda ,\xi)\in [0,\infty )\times C_T^1(N)$ with $\xi$ a $T$-periodic solution of 
\eqref{uno} and of \eqref{due}, respectively. Recall that a solution pair $(\lambda ,\xi)$ 
of \eqref{uno} is said to be \emph{trivial} if $\lambda=0$ and $\xi$ is constant, whereas
a solution pair $(\lambda ,\xi)$ of \eqref{due} is \emph{trivial} merely if $\lambda=0$.

For the sake of simplicity  we make some conventions. We will regard every space 
as its image in the following diagram of natural inclusions 
\begin{equation}
\begin{CD}
N & @>>> & C_T^1(N) \\ 
@VVV &  & @VVV \\ 
\left[ 0,\infty \right) \times N & @>>> & \left[ 0,\infty \right)
\times C_T^1(N)
\end{CD}
\label{diag}
\end{equation}
In particular, we will identify $N$ with its image in $C_T^1(N)$ under the embedding which 
associates to any $\xi_0\in N$ the map $\hat \xi_0\in C_T^1(N)$ constantly equal to $\xi_0$. 
Moreover we will regard $N$ as the slice $\{0\}\times N\subset [0,\infty)\times N$ and,
analogously, $C_T^1(N)$ as $\{0\}\times C_T^1(N)$. We point out that the images of the above 
inclusions are closed. According to these identifications, if $\Omega$ is an open subset of 
$[0,\infty)\times C_T^1(N)$, by $ N\cap\Omega$ we mean the open subset of $N$ given by all 
$\xi_0\in N$ such that the pair $(0,\hat\xi_0)$ belongs to $\Omega$. If $\mathcal{O}$ is an 
open subset of $[0,\infty)\times N$, then $\mathcal{O}\cap N$ represents the open set 
$\big\{\xi_0\in N:(0,\xi_0)\in\mathcal{O}\big\}$.
\medskip 

We need to introduce some further notation. 
Given $f\colon TN\to\R^k$ tangent to the manifold
$N$, we define the tangent vector field $f|_N\colon N\to\R^k$ given by $f|_N(\xi)=f(\xi,0)$ for
any $\xi\in N$. The following results of \cite{FP93} and of \cite{FS99} play a central role:

\begin{theorem}[\cite{FS99} Th.\ 4.2]\label{ramo2}
Let $N\subseteq\R^k$ be a boundaryless manifold, and let $f\colon TN\to\R^k$ and 
$h\colon\R\times TN\to\R^k$ be tangent to $N$. Assume that $h$ is $T$-periodic in the first 
variable. Let $\Omega$ be an open subset of $[0,\infty)\times C_T^1(N)$.
Assume that 
$\deg\big(f|_N, N\cap\Omega\big) $ is well defined and nonzero.
Then, $\Omega$ contains a 
connected set $\Gamma$ of nontrivial solution pairs for \eqref{uno} whose closure 
meets $\big(f|_N\big)^{-1}(0)\cap\Omega$ and is not contained in any compact subset of 
$\Omega$. In particular, if $N$ is closed in $\R^k$ and $\Omega=[0,\infty)\times C_T^1(N)$, 
then $\Gamma$ is unbounded.
\end{theorem}

Given $h\colon\R\times TN\to\R^k$ tangent to the manifold $N$, we define the
\emph{mean value} tangent vector 
field $w_h\colon N\to\R^k$ given by 
\begin{equation}\label{mean}
 w_h (\xi)=\frac{1}{T} \int_0^T h(t,\xi,0)dt.
\end{equation}

\begin{theorem}[\cite{FP93} Th.\ 2.2]\label{ramo1}
Let $N\subseteq\R^k$ be a boundaryless manifold, and let $h\colon\R\times TN\to\R^k$ be 
tangent to $N$ and $T$-periodic in the first variable. Let $\Omega$ be an open subset of 
$[0,\infty)\times C_T^1(N)$.
Assume that $\deg\big(w_h,N\cap\Omega\big) $ is well defined 
and nonzero.
Then, $\Omega$ contains a connected set $\Gamma$ of nontrivial solution pairs for 
\eqref{due} whose closure meets $ w_h^{-1}(0)\cap\Omega$ and is not contained 
in any compact subset of $\Omega$. In particular, if $N$ is closed in $\R^k$ and 
$\Omega=[0,\infty)\times C_T^1(N)$, then $\Gamma$ is unbounded.
\end{theorem}

\section{Main results}

Throughout this section $U$ will be an open and connected subset of $\R^m\times\R^s$. {}From 
now on, we will denote by $\Pr1$ the projection onto the first factor in the 
Cartesian product $\R^m\times\R^s$. 

We will always assume that $g\colon U\to\R^s$ is a $C^\infty$ function such that 
$\partial_2g(x,y)$ is nonsingular for any $(x,y)\in U$, and that $M=g^{-1}(0)$. Moreover we 
will identify $\xi=(x,y)\in U$. It will also be convenient, given a continuous tangent vector 
field, $f\colon TM\to\R^m\times\R^s$, to denote by $\t f$ an arbitrary extension of the map 
$f|_M\colon (x,y)\mapsto f(x,y,0,0)$ to $U$ (see Remark \ref{remext}) and to let 
$\Pr1\t f(x,y)$ be the projection of $\t f(x,y)$ on $\R^m$ for any $(x,y)\in U$.

\begin{theorem}\label{tuno}
Let $f\colon TM\to\R^m\times\R^s$ and $h\colon\R\times TM\to\R^m\times\R^s$ be continuous 
tangent vector fields, with $h$ of a given period $T>0$ in the first variable.  Define 
$\F\colon U\to\R^m\times\R^s$ by $\F(x,y)=\big(\Pr1\t f(x,y),g(x,y)\big)$ for any 
$\xi=(x,y)\in U$. Given an open set $\Omega\subseteq[ 0,\infty)\times C_T^1(M)$, let 
$\Oc\subseteq\R^m\times\R^s$ be open with the property that $\Oc\cap M=M\cap\Omega$. 
Assume that $\deg (\F,\Oc)$ is well defined and nonzero. Then there exists a connected 
set $\Gamma$ of nontrivial solution pairs for \eqref{uno} in $\Omega$ whose closure in 
$\Omega$ meets $(f|_M)^{-1}(0,0) \cap\Omega$ and is not compact. In particular, if $M$ is 
closed in $\R^m\times\R^s$ and $\Omega=[0,\infty)\times C_T^1(M)$, then $\Gamma$ is 
unbounded.
\end{theorem}

\begin{proof}
By Theorem \ref{formuladeg} we have
\[
|\deg(f|_M,M\cap\Omega)|=|\deg(f|_M,\Oc\cap M)|=|\deg(\F,\Oc)|.
\]
Thus, $\deg(f,M\cap\Omega)\neq 0$ and the assertion follows {}from Theorem \ref{ramo2}.
\end{proof}

\bigskip
The following result concerns Equation \eqref{due}. 
As in the previous section, given $h\colon\R\times TM\to\R^k$ tangent to $M$,
we define the tangent vector field $w_h$ on $M$ by \eqref{mean}.
Moreover by $\Pr1 \t w_h(x,y)$ we denote the projection on $\R^m$ of an
arbitrary extension $\t w_h(x,y)$ of the map $w_h(x,y)$  to $U$.

\begin{theorem}\label{tdue}
Let $h\colon \R\times TM\to\R^m\times\R^s$ be a continuous tangent vector field, of a given 
period $T>0$ in the first variable. Define $\Phi\colon U\to\R^m\times\R^s$ by 
$\Phi(x,y)=\big(\Pr1 \t w_h(x,y),g(x,y)\big)$ for any $\xi=(x,y)\in U$. Given an open 
set $\Omega\subseteq [0,\infty)\times C_T^1(M)$, let $\Oc\subset\R^m\times\R^s$ be an 
open subset with the property that $\Oc\cap M=M\cap\Omega$. Assume that $\deg(\Phi,\Oc)$ 
is well defined and nonzero. Then there exists a connected set $\Gamma$ of nontrivial 
solution pairs for \eqref{due} in $\Omega$ whose closure in $\Omega$ is not compact and 
meets the set $w_h^{-1}(0)\cap\Omega$. In particular, if $M$ is closed in 
$\R^m\times\R^s$ and $\Omega=[0,\infty)\times C_T^1(M)$, then $\Gamma $ is unbounded.
\end{theorem}

\begin{proof}
By Theorem \ref{formuladeg} we have
\[
|\deg(w_h,M\cap\Omega)|=|\deg(w_h,\Oc\cap M)|=|\deg(\Phi,\Oc)|.
\]
Thus, 
$\deg(w_h, M\cap\Omega)\neq 0$ and the assertion follows {}from Theorem \ref{ramo1}.
\end{proof}

\begin{remark}\label{remfeF}
 It is worth noticing that, in Theorem \ref{tuno}, a point $(x,y)\in\R^m\times\R^s$
belongs to $(f|_M)^{-1}(0,0)$ if and only if $\F(x,y)=0$. Similarly, in Theorem 
\ref{tdue}, $(x,y)\in w_h^{-1}(0)$ if and only if $\Phi(x,y)=0$.
\end{remark}

\begin{example}\label{ex.gravita}
Consider the simple case when $m=s=1$, $U=\R\times\R$ and 
$g(x,y)=\frac{1}{3}(x^2+1)y+\frac{1}{27}y^3+x$. Assume that a point $\PP$ of mass $\mm$ 
is constrained without friction to the curve $M=g^{-1}(0)$ lying on a vertical plane 
(see figure \ref{fig.essegiu}) and acted upon by the gravitational force 
$\varphi(x,y)=(0,-\mm\gb)^T$. Here, $\gb$ denotes the gravitational constant. The motion 
of $\PP$ is determined by the following second order equation on $M$:
\[
\mm\ddot\xi_\pi=\Prj_\xi \varphi(\xi)
\]
where $(x,y)$ and $\Prj_\xi$ denotes the orthogonal projection of $\R^2$ onto $T_\xi M$. 
In order to apply Theorems \ref{tuno} or \ref{tdue} it is necessary to compute 
$\Prj_\xi\varphi(\xi)$. This is readily done, in fact,
\[
f(\xi):=\Prj_\xi\varphi(\xi)
       =\varphi(\xi)- \frac{\left<\varphi(\xi),\nabla g(\xi)\right>}{\|\nabla g(\xi)\|^2}\nabla g(\xi),
\]
where `$\left<\cdot,\cdot\right>$' denotes the scalar product. In coordinates, we get the rather 
unappealing formula below:
\[
f(x,y)=
\frac{3\mm\gb}{(6xy+9)^2+(y^2+3x^2+3)^2} 
\begin{pmatrix}
(y^2+3x^2+3)(2xy+3)\\
-3(2xy+3)
\end{pmatrix}.
\]

\begin{figure}[ht!]
  \centering
    \psset{xunit=1cm,yunit=1cm }
    \begin{pspicture}(-6, -1.7)(6,2) 
     \uput[u](3.9,0){\scriptsize $x$}
     \uput[l](0,1.8){\scriptsize $y$}
       \parametricplot[linecolor=gray,plotpoints=150,linewidth=0.3pt]
             {1}{10}{t -1 add t t mul 1 add div sqrt t mul neg   t -1 add t t mul 1 add div sqrt 3 mul }
       \parametricplot[linecolor=gray,plotpoints=150,linewidth=0.3pt]
             {1}{10}{t -1 add t t mul 1 add div sqrt t mul  t -1 add t t mul 1 add div sqrt 3 mul neg }
     \psline[linewidth=0.4pt]{->}(-4, 0)(4, 0) %
     \psline[linewidth=0.4pt]{->}(0,-1.6)(0,1.9)
      \psline[linewidth=0.6pt]{->}(2.4,1.4)(2.4,0.4)
      \uput[r](2.4,0.8){\scriptsize gravity}
 \uput[ur](-2.2,1.2){\scriptsize $\xi(t)$}
     \psdot(-1.96,1.17)
     \psline{->}(-1.96,1.17)(-1.96,0.17) 
     \psline{->}(-1.96,1.17)(-2.23,1.09)
     \psline[linestyle=dotted]{*-o}(1.0987,0)(1.0987,-1.3652)
     \uput[u](1.0987,0){\scriptsize $\frac{1}{2}\sqrt{2+2\sqrt{2}}$}
     \psline[linestyle=dotted]{*-o}(0,-1.3652)(1.0987,-1.3652)
     \uput[l](0,-1.3652){\scriptsize $-\frac{3}{\sqrt{2+2\sqrt{2}}}$}
     \psline[linestyle=dotted]{*-o}(0,1.3652)(-1.0987,1.3652)
     \uput[r](0,1.3652){\scriptsize $\frac{3}{\sqrt{2+2\sqrt{2}}}$}
     \psline[linestyle=dotted]{*-o}(-1.0987,0)(-1.0987,1.3652)
     \uput[d](-1.0987,0){\scriptsize $-\frac{1}{2}\sqrt{2+2\sqrt{2}}$}
     \uput*[ur](-1.96,0.35){\scriptsize $\varphi\big(\xi(t)\big)$}
     \uput[l](-2,1.27){\scriptsize $f\big(\xi(t)\big)$}
    \end{pspicture}
  \caption{The setting of Example \ref{ex.gravita}.}\label{fig.essegiu}
\end{figure}
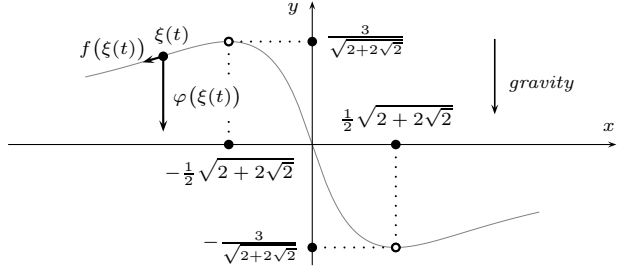
Simple considerations show that there are exactly two zeros of $f(x,y)$ in $U$ corresponding 
to the two real roots of the polynomial $4x^4-4x^2-1$, that is $x=\pm\frac{1}{2}\sqrt{2+2\sqrt{2}}$.
Now, let $\Oc=(-\infty,0)\times\R$, $\Omega=[0,\infty)\times C^1_T\big(\Oc\big)$, and
\[
 \F(x,y)=\left(\frac{3\mm\gb(y^2+3x^2+3)(2xy+3)}{(6xy+9)^2+(y^2+3x^2+3)^2}\,,\, 
              \frac{1}{3}(x^2+1)y+\frac{1}{27}y^3+x\right).
\]
A lengthy but straightforward computation shows that $\deg(\F,\Oc)=1$. Hence, Theorem \ref{tuno} 
yields that if any $T$-periodic force $\lambda h(t,\xi,\dot \xi)$, $\lambda\geq 0$, is added to 
$f$ there exists an unbounded connected set $\Gamma$ of nontrivial solution pairs for \eqref{uno} 
in $\Omega$ whose closure in $\Omega$ meets $f^{-1}(0)\cap\Omega$.
\end{example}

\begin{figure}[ht!]
  \centering
    \psset{xunit=0.6cm,yunit=0.6cm }
    \begin{pspicture}(-4, -2.5)(4,3) 
     \psline[linewidth=0.5pt]{->}(-3.7, 0)(3.7, 0) %
     \psline[linewidth=0.5pt]{->}(0,-2.5)(0,3)
     \uput[u](3.6,0){\scriptsize $x$}
     \uput[l](0,2.4){\scriptsize $y$}
     \parabola[linewidth=0.8pt](3,2.5)(0,-2)
     \pscoil[coilarm=.5cm,linewidth=0.7pt,coilwidth=.25cm]{*-*}(0,0)(2.5,1.125)
     \uput[ul](2.5,1.125){\scriptsize $\PP$}
     \uput[ul](0,0){\scriptsize $O$}
     \psline{->}(2.5,1.125)(1.5,-1.375)
     \uput[r](1.5,-1.375){\scriptsize $f$}
    \end{pspicture}
  \caption{The setting of Example \ref{ex.parabolamolla}.}\label{fig.parabolamolla}
\end{figure}
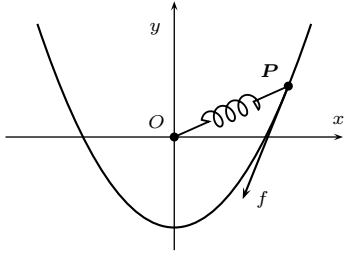
\begin{example}\label{ex.parabolamolla} 
Consider a unit mass point $\PP$ constrained without friction to the parabola $M$ of Example 
\ref{parabola1}, of equation $\frac{1}{2}x^2-y-2=0$ lying on an horizontal plane (so that 
the weight has no effect on the motion). 
Assume that the extremities of a spring with elastic coefficient $\kk$ and negligible mass are attached 
one to $\PP$ and the other to the origin $O$. (See Figure \ref{fig.parabolamolla}.) It is easy to 
verify that the active force acting on $\PP$ is given by
\[
 f(x,y) = -\frac{\kk(y+1)}{x^2+1}\begin{pmatrix}
                             x\\
                             x^2
                            \end{pmatrix}.
\]
Clearly $f|_M$ is a vector field tangent to $M$, and the motion of $\PP$ is described by the  
equation $\ddot\xi_\pi=f(\xi)$ on $M$.
Here again $m=s=1$ and $U=\R\times\R$. 
Let
\[
 \F(x,y)=\left(-\frac{\kk x(y+1)}{x^2+1},\frac{1}{2}x^2-y-2\right),
\]
which is clearly defined on $U$,
and let $\Omega=[0,\infty)\times C_T^1(M)$. Since $\deg(\F,U)=1$, Theorem \ref{tuno} shows that 
if any $T$-periodic force $\lambda h(t,\xi,\dot \xi)$, $\lambda\geq 0$, is added to $f$, then 
there exists an unbounded connected set $\Gamma$ of nontrivial solution pairs for \eqref{unoedue} 
in $\Omega$ whose closure in $\Omega$ meets 
$f^{-1}(0)\cap\Omega=\{(-\sqrt{2},-1),(0,-2),(\sqrt{2},-1)\}$.
\end{example}

\section{Applications to a class of Differential-Algebraic Equations} 

Let us now consider applications to semi-explicit DAEs of the 
following forms on an open set $U\subseteq\R^m\times\R^s$:

\begin{equation}\label{ord2dae1}
 \left\{
  \begin{array}{l}
   \ddot x = \f(x,y,\dot x,\dot y)+\lambda \h(t,x,y,\dot x,\dot y),\\
   \g(x,y)=0,
  \end{array}\right.
\end{equation}
and 
\begin{equation}\label{ord2dae2}
 \left\{
  \begin{array}{l}
   \ddot x = \lambda \h(t,x,y,\dot x,\dot y),\\
   \g(x,y)=0,
  \end{array}\right.
\end{equation}
where we assume that $\f\colon TU\to\R^m\times\R^s$ and $\h\colon\R\times TU\to\R^m\times\R^s$ 
are continuous maps, $\h$ is $T$-periodic in the first variable, and $\g\colon U\to\R^s$ is 
$C^\infty$ and such that $\partial_2\g(x,y)$ is invertible for all $(x,y)\in U$.

Let us discuss briefly the notion of solution of DAEs as \eqref{ord2dae1} and \eqref{ord2dae2}.
In general, let $\E\colon\R\times TU\to\R^m$ be a continuous map and let $\g$ be as in 
\eqref{ord2dae1} and \eqref{ord2dae2}. Consider the equation
\begin{equation}\label{gendae}
  \left\{
  \begin{array}{l}
   \ddot x = \E(t,x,y,\dot x,\dot y),\\
   \g(x,y)=0.
  \end{array}\right.
\end{equation}
By a \emph{solution} of \eqref{gendae} we mean a pair of $C^1$ functions $x\colon J\to\R^k$ 
and $y\colon J\to\R^s$, $J$ an interval, with the property that 
$\ddot x(t)=\E(t,x(t),y(t),\dot x(t),\dot y(t))$ and $\g\big(x(t),y(t)\big)=0$ for all $t\in J$. 
Differential-Algebraic Equations as \eqref{gendae} are closely related to second order 
differential equations on manifolds as described in Section \ref{secsec}.
It is now our aim to clarify such a connection.

Let $(x,y)$ be a solution of \eqref{gendae}. Differentiating twice the relation 
$\g\big(x(t),y(t)\big)=0$ with respect to $t$, we get
\[
 \D_1\g\big(x(t),y(t)\big)\ddot x(t)+\D_2\g\big(x(t),y(t)\big)\ddot y(t)
           +\g''\big(x(t),y(t)\big)\big(\dot x(t),\dot y(t)\big)=0 .
\]

Thus, recalling that $\ddot x(t) = \E\big(t,x(t),y(t),\dot x(t),\dot y(t)\big)$, we get
\begin{multline*}
\ddot y(t)
   =-\big[\D_2\g\big(x(t),y(t)\big)\big]^{-1}
             \Big(\D_1\g\big(x(t),y(t)\big)\E\big(t,x(t),y(t),\dot x(t),\dot y(t)\big)\\
                             +\g''\big(x(t),y(t)\big)\big(\dot x(t),\dot y(t)\big)\Big) .
\end{multline*}

For $t\in\R$, $p,u\in\R^m$ and $q,v\in\R^s$ (hence $(p,q,u,v)\in TU$),
let

\[
 \t\E(t,p,q,u,v)=\begin{pmatrix}
           \E(t,p,q,u,v)\\
-\big[\D_2\g(p,q)\big]^{-1}\Big(\D_1\g(p,q)\E\big(t,p,q,u,v\big)
                                             +\g''(p,q)\big(u,v\big)\Big)
                 \end{pmatrix}
\]

The following proposition shows that DAEs as \eqref{gendae} can be regarded as second order 
ODEs on differential manifolds as discussed in Section \ref{secsec}.

\begin{proposition}\label{propeq}
Equation \eqref{gendae} is equivalent to the following second order ODE on $M=\g^{-1}(0)$:
\begin{equation}\label{prjgendae}
 \ddot \xi_\pi = \Prj_{\xi}\t\E(t,\xi,\dot \xi),
\end{equation}
in the sense that $\big(x(t),y(t)\big)$ is a solution of \eqref{gendae} if and only if
$\xi(t)=\big(x(t),y(t)\big)$ is a solution of \eqref{prjgendae}. Here $\Prj_{(p,q)}$ denotes
the orthogonal projection of $\R^m\times\R^s$ onto $T_{(p,q)}M$.
\end{proposition}

\begin{proof}
Clearly any solution of \eqref{gendae} that meets $M$ satisfies also \eqref{prjgendae}. 

Let us prove the converse. According to Equation \eqref{prjxi} one has that for any $M$-valued 
$C^2$ curve $t\mapsto\xi(t)$ one has that, for any $t$,
\[
 \Prj_{\xi(t)}^\perp\ddot\xi(t)=r\big(\xi(t),\dot\xi(t)\big)
\]
where $\Prj_{\xi(t)}^\perp$ denotes the projection of the ambient space $\R^k=\R^m\times\R^s$ 
onto the orthogonal complement $(T_{\xi(t)}M)^\perp$ of $T_{\xi(t)}M$, and $r\colon TM\to\R^k$ 
is the map given by
\[
 r\big(\xi,\eta\big)=
     -\left(\g'(\xi)|_{T_\xi M^\perp}\right)^{-1}\g''(\xi)(\eta,\eta)\in (T_\xi M)^\perp
\]
for all $\xi\in M$ and $\eta\in T_\xi M$. Let us now see what relation exists between $r$ 
and $\t\E$. Observe that, given any $M$-valued curve $\xi$ for which 
$\ddot\xi(t)\equiv\t\E\big(t,\xi(t),\dot\xi(t)\big)$, applying $\Prj_{\xi(t)}^\perp$ on both 
sides of this equality we get
\[
 r\big(\xi(t),\dot\xi(t)\big)=\Prj_{\xi(t)}^\perp\ddot\xi
                                 =\Prj_{\xi(t)}^\perp\t\E\big(t,\xi(t),\dot\xi(t)\big),
\]
whenever defined. The arbitrariness of $\xi$ shows that
\[
 r(\xi,\eta)=\Prj_{\xi}^\perp\t\E(t,\xi,\eta)
\]
for all $\xi\in M$, $\eta\in T_\xi M$ and $t\in\R$.

Let now $\xi\colon J\subseteq\R\to U$ be a solution of \eqref{prjgendae}. 
From equation
\eqref{phipiur} we get that, whenever defined,
\[
 \ddot\xi(t)=\Prj_{\xi(t)}\t\E\big(t,\xi(t),\dot\xi(t)\big)+r\big(\xi(t),\dot\xi(t)\big)
\]
Hence, 
\[
 \ddot\xi(t)=\Prj_{\xi(t)}\t\E\big(t,\xi(t),\dot\xi(t)\big)
                       +\Prj_{\xi(t)}\t\E\big(t,\xi(t),\dot\xi(t)\big)
            =\t\E\big(t,\xi(t),\dot\xi(t)\big),
\]
and the assertion is proved.
\end{proof}

We observe that
the equivalence of \eqref{gendae} and \eqref{prjgendae} could be proven also using the 
explicit expression for $r$ found in Section \ref{secsec} (see formula \eqref{explicit}).

\smallskip
The equivalence shown by Proposition \ref{propeq} allows us to apply the results of the 
previous sections to Equations \eqref{ord2dae1} and \eqref{ord2dae2} by the means of 
the corresponding ODEs on $M$.  

Let $\f$, $\g$, and $\h$ be as above. Define the maps $\Psi\colon T(U)\to\R^m\times\R^s$ 
and $\Upsilon\colon\R\times T(U)\to\R^m\times\R^s$ by 
\begin{multline*}
\Psi(p,q,u,v)=\\
  \left(\f(p,q,u,v),
      -\big[\D_2\g(p,q)\big]^{-1}\Big(\D_1\g(p,q)\f\big(p,q,u,v\big)
                                            +\g''(p,q)\big(u,v\big)\Big)\right),
\end{multline*}
and
\begin{multline*}
\Upsilon(t,p,q,u,v)=\\
\left(
   \h(t,p,q,u,v),
     -\big[\D_2\g(p,q)\big]^{-1}\Big(\D_1\g(p,q)\h\big(t,p,q,u,v\big)
                                              +\g''(p,q)\big(u,v\big)\Big)\right).
\end{multline*}
We have that \eqref{ord2dae1} and \eqref{ord2dae2} are equivalent, respectively, to
\begin{equation}\label{ord2ode1}
\ddot\xi_\pi=\Prj_\xi\Psi(\xi,\dot \xi)+\lambda\Prj_\xi\Upsilon(t,\xi,\dot\xi),
\end{equation}
and
\begin{equation}\label{ord2ode2}
\ddot\xi_\pi=\lambda\Prj_\xi\Upsilon(t,\xi,\dot\xi),
\end{equation}
where $\xi=(x,y)$.

We now use Equations \eqref{ord2ode1} and \eqref{ord2ode2} to derive, by the means 
of Theorems \ref{tuno} and \ref{tdue}, results concerning the sets of solution pairs 
for Equations \eqref{ord2dae1} and \eqref{ord2dae2}. 

Let us consider first Equation \eqref{ord2dae1}. We define $\f_0(p,q):=\f(p,q,0,0)$ 
and
\[
\Psi_0(p,q):=\Psi(p,q,0,0)=\left(\f_0(p,q), 
                  -\big[\D_2\g(p,q)\big]^{-1}\D_1\g(p,q)\f_0\big(p,q\big)\right).
\]
Observe that $\Psi_0$ is tangent to $M$ in the sense that
$\Psi_0(p,q)\in T_{(p,q)}M$ for any $(t,p,q)\in\R\times M$.
Hence, 
\[
\Pr1\Prj_{(p,q)}\Psi_0(p,q)=\f_0(p,q).
\]
Then {}from Theorem \ref{tuno} and Remark \ref{remfeF} we get the following.

\begin{corollary}\label{tunoB}
Let $\f$, $\h$, $\g$ and $U$ be as above. For any $(x,y)\in U$, define 
$\F\colon U\to\R^m\times\R^s$ by 
\[
\F(x,y)=\big(\f_0(x,y),\g(x,y)\big).
\]
Let $\Omega\subseteq[0,\infty)\times C_T^1(\R^m\times\R^s)$ be open. Assume that 
$\deg(\F, U\cap\Omega)$ is well defined and nonzero. Then there exists a connected set 
$\Gamma$ of nontrivial solution pairs for \eqref{ord2dae1} in $\Omega$ whose closure in 
$\Omega$ meets $\F^{-1}(0)\cap\Omega$ and is not compact. In particular, if $\g^{-1}(0)$ 
is a closed subset of $\R^m\times\R^s$ and $\Omega=[0,\infty)\times C_T^1(\R^m\times\R^s)$, 
then $\Gamma$ is unbounded.
\end{corollary}

A similar argument allows us to deduce the following consequence of Theorem \ref{tdue}.
We focus on Equation \eqref{ord2dae2}. We define $\h_0(t,p,q):=\h(t,p,q,0,0)$ and
the tangent vector field $w_{\h_0}$ on $M$ as in \eqref{mean} with $\xi$ in place of 
$(p,q)$.

Since, clearly, $w_{\Prj_{(p,q)}\Upsilon}\in T_{(p,q)}M$ for any $(p,q)\in U$, one has 
that
\[
\Pr1w_{\Prj_{(p,q)}\Upsilon}(p,q)=w_{\h_0}(p,q).
\]
Hence {}from Theorem \ref{tdue} and Remark \ref{remfeF} we get

\begin{corollary}\label{tunoC}
Let $\h$ and $\g$ be as above. For any $(x,y)\in U$, define $\F\colon U\to\R^m\times\R^s$ by 
\[
\F(x,y)=\big(w_{\h_0}(x,y),\g(x,y)\big).
\]
Let $\Omega\subseteq[0,\infty)\times C_T^1(\R^m\times\R^s)$ be open. Assume that 
$\deg(\F,U\cap\Omega)$ is well defined and nonzero. Then there exists a connected set 
$\Gamma$ of nontrivial solution pairs for \eqref{ord2dae2} in $\Omega$ whose closure in 
$\Omega$ meets $\F^{-1}(0)\cap\Omega$ and is not compact. In particular, if $\g^{-1}(0)$ 
is a closed subset of $\R^m\times\R^s$ and $\Omega=[0,\infty)\times C_T^1(\R^m\times\R^s)$, 
then $\Gamma$ is unbounded.
\end{corollary}

We conclude this section with a few illustrative examples.

\begin{example}\label{parabola2}
 Consider the setting of Example \ref{ex.parabolamolla}. The horizontal acceleration experienced by
$\PP$ is $-\kk\frac{(y+1)x}{x^2+1}$. Hence, the dynamics of Example \ref{ex.parabolamolla} is 
described also by the following DAE:
\[
 \left\{
     \begin{array}{l}
        \ddot x = -\kk\frac{(y+1)x}{x^2+1},\\
        \frac{1}{2}x^2-y-2=0.
     \end{array}
\right.
\]

Let us perturb this equation as follows
\[
 \left\{
     \begin{array}{l}
        \ddot x = -\kk\frac{(y+1)x}{x^2+1}+\lambda \h(t,x,y,\dot x,\dot y),\\
        \frac{1}{2}x^2-y-2=0,
     \end{array}
\right.
\]
where $\h\colon\R\times\R^2\times\R^2\to\R$ is a continuous $T$-periodic function.
Let $\Omega=[0,\infty)\times C_T^1(\R^m\times\R^s)$ and put 
\[
 \F(x,y)=\left(-\kk\frac{(y+1)x}{x^2+1},\frac{1}{2}x^2-y-2 \right),
\]
as in Example \ref{ex.parabolamolla}. Since $\deg(\F,U\cap\Omega)=1\neq 0$,
Corollary \ref{tunoB} imply that in $\Omega$
there exists an unbounded connected set $\Gamma$ of nontrivial solution pairs for the 
perturbed equation whose closure meets 
\[
\F^{-1}(0)\cap\Omega=\left\{\big(-\sqrt{2},-1\big),(0,-2),\big(\sqrt{2},-1\big)\right\}. 
\]
The result is exactly the same as in Example \ref{ex.parabolamolla}, but now we do not 
need the second component of the active force, nor we need to bother about second order 
differential equations on manifolds. Notice also that the result holds true for 
any choice of the function $\h$.
\end{example}

\begin{example}
Let $m=s=1$ and consider the following DAE (the constraint is that of Example \ref{parabola1}):
\[
 \left\{
     \begin{array}{l}
        \ddot x = \lambda\big(x+\sin(t) y\big),\\
        \frac{1}{2}x^2-y^3-y=0.
     \end{array}
\right.
\]
Let $\Omega=[0,\infty)\times C_T^1(\R^m\times\R^s)$ and put
\[
 \F(x,y)=\left( \frac{1}{2\pi}\int_0^{2\pi}\big(x+\sin(t) y\big)\dif t\,,
                                                         \,\frac{1}{2}x^2-y^3-y\right)
        =\left(x\,,\,\frac{1}{2}x^2-y^3-y\right)
\]
Since $\deg(\F,U\cap\Omega)=-1$, by Corollary \ref{tunoC}, we get that in $\Omega$ there 
exists an unbounded connected set $\Gamma$ of nontrivial solution pairs for this equation 
whose closure in $\Omega$ meets $\{(0,0)\}$ (regarded as a a solution pair).
\end{example}

\begin{example}
Let $m=1$, $k=3$ and $U=\R^3$, and consider the following DAE (the constraint is that of 
Example \ref{mostro}):
 \[
 \left\{
     \begin{array}{l}
        \ddot x = x-2y+\lambda \h(t,x,y,z,\dot x,\dot y,\dot z),\\
        z^3+z-x=0,\\
        z-y+x^2=0,
     \end{array}
\right.
\]
where $\h\colon\R\times\R^3\times\R^3\to\R$ is any continuous $T$-periodic perturbing 
function. Let $\Omega=[0,\infty)\times C_T^1(\R^1\times\R^2)$ and put
\[
 \F(x,y)=\left(x-2y,z^3+z-x, z-y+x^2\right).
\]
As it is readily checked $\F^{-1}(0,0,0)=\{(0,0,0)\}$ and $\deg(\F,U\cap\Omega)=-1$. 
Hence, for any `perturbing' function $\h$, Corollary \ref{tunoC} yields an unbounded 
connected set $\Gamma$ in $\Omega$ of nontrivial solution pairs for this equation whose 
closure in $\Omega$ meets the point $\{(0,0,0)\}$ (regarded as a solution pair). Notice 
also, as in Example \ref{parabola2}, that the above statement holds true for any choice 
of the function $\h$.
\end{example}

\end{document}